\documentclass{amsart}
\usepackage{amsmath}
\usepackage{amssymb}
\usepackage{amscd}
\usepackage{amstext}
\usepackage[all]{xy}
\xyoption{2cell}
\UseTwocells
\usepackage{graphicx}
\usepackage{color}

\newcommand{\calL}{{\mathcal L}}
\newcommand{\calP}{{\mathcal P}}
\newcommand{\calI}{{\mathcal I}}
\newcommand{\calR}{{\mathcal R}}
\newcommand{\calS}{{\mathcal S}}
\newcommand{\calU}{{\mathcal U}}
\newcommand{\iso}{\cong}
\newcommand{\niso}{\not\cong}
\newcommand{\bbN}{{\mathbb N}}
\newcommand{\bbP}{{\mathbb P}}
\newcommand{\bbR}{{\mathbb R}}
\newcommand{\bbZ}{{\mathbb Z}}
\newcommand{\bdx}{\boldsymbol{x}}
\newcommand{\ds}{\oplus}
\newcommand{\Ds}{\bigoplus}
\newcommand{\DS}{\bigoplus\limits}
\newcommand{\Sm}{\sum}
\newcommand{\supp}{\operatorname{supp}}
\renewcommand{\mod}{\operatorname{mod}}
\newcommand{\rank}{\operatorname{rank}}
\newcommand{\Hom}{\operatorname{Hom}}
\newcommand{\End}{\operatorname{End}}
\newcommand{\rep}{\operatorname{rep}}
\newcommand{\rad}{\operatorname{rad}}
\newcommand{\soc}{\operatorname{soc}}
\newcommand{\Ker}{\operatorname{Ker}}
\newcommand{\Cok}{\operatorname{Coker}}
\newcommand{\im}{\operatorname{Im}}
\newcommand{\blank}{\operatorname{-}}
\newcommand{\udim}{\operatorname{\underline{dim}}\nolimits}
\newcommand{\al}{\alpha}
\newcommand{\be}{\beta}
\newcommand{\Ga}{\Gamma}
\newcommand{\la}{\lambda}
\newcommand{\La}{\Lambda}
\newcommand{\ta}{\tau}
\newcommand{\J}{\mathsf{J}}
\newcommand{\inv}{^{-1}}

\newcommand{\ya}[1]{\xrightarrow{#1}}
\newcommand{\Mat}[1]{\mathbb{M}_{#1}(\Bbbk)}

\newcommand{\bmat}[1]{\begin{bmatrix} #1 \end{bmatrix}}

\newcommand{\smat}[1]{\begin{smallmatrix} #1 \end{smallmatrix}}

\newcommand{\Ma}{ \begin{matrix}1&0\\ 0&1 \\0&0\end{matrix}}
\newcommand{\Mb}{ \begin{matrix}0&0\\1&0\\ 0&1 \end{matrix}}
\newcommand{\Maa}{\smat{M(\al)&&&\\&M(\al)&&\\&&\ddots &\\&&&M(\al)}}
\newcommand{\Mab}{\smat{-E_{d_2}&&&&\rule{10pt}{0pt}0\rule{10pt}{0pt}\\&-E_{d_2}&&&0\\&&\ddots &&\vdots\\&&&-E_{d_2}&0}}
\newcommand{\Mba}{\smat{M(\be)&&&\\&M(\be)&&\\&&\ddots &\\&&&M(\be)}}
\newcommand{\Mbb}{\smat{\rule{10pt}{0pt}0\rule{10pt}{0pt}&-E_{d_2}&&&\\0&&-E_{d_2}&&\\\vdots&&&\ddots &\\0&&&&-E_{d_2}}}

\newcommand{\rej}{\operatorname{\sf Rej}}
\newcommand{\tr}{\operatorname{\sf Tr}}
\newcommand{\bd}{\boldsymbol{d}}

\newtheorem{theorem}{Theorem}[section]
\newtheorem{proposition}[theorem]{Proposition}
\newtheorem{lemma}[theorem]{Lemma}

\theoremstyle{definition}
\newtheorem{definition}[theorem]{Definition}
\newtheorem{example}[theorem]{Example}

\theoremstyle{remark}
\newtheorem{remark}[theorem]{Remark}

 \makeatletter
    
    \@addtoreset{equation}{section}
  \makeatother

\title{Decomposition theory of modules: the case of Kronecker algebra}
\author{Hideto Asashiba \and Ken Nakashima \and Michio Yoshiwaki}

\address{
\begin{flushleft}
Hideto Asashiba \\
Department of Mathematics, Faculty of Science, Shizuoka University,\\
836 Ohya, Suruga-ku, Shizuoka, 422-8529, Japan.\\
\end{flushleft}
}
\email{asashiba.hideto@shizuoka.ac.jp}
\address{
\begin{flushleft}
Ken Nakashima\\ 
Graduate School of Science and Technology, Shizuoka University,\\
836 Ohya, Suruga-ku, Shizuoka, 422-8529, Japan.\\
Tel.: +81-54-238-4722 \\ Fax: +81-54-238-4722\\
\end{flushleft}
}
\email{gehotan@gmail.com}
\address{
\begin{flushleft}
Michio Yoshiwaki \\ 
Department of Mathematics, Faculty of Science, Shizuoka University,\\
836 Ohya, Suruga-ku, Shizuoka, 422-8529, Japan.\\
Osaka City University Advanced Mathematical Institute,\\
3-3-138 Sugimoto, Sumiyoshi-ku, Osaka, 558-8585, Japan.\\
\end{flushleft}
}
\email{yoshiwaki.michio@shizuoka.ac.jp}


\begin{document}
\maketitle

\begin{abstract}
Let $A$ be a finite-dimensional algebra over an algebraically closed field $\Bbbk$.  For any finite-dimensional $A$-module $M$ we give a general formula that computes the indecomposable decomposition of $M$ without decomposing it,
for which we use the knowledge of AR-quivers that are already computed in many cases.
The proof of the formula here is much simpler than that in a prior literature by Dowbor and Mr{\'o}z.
As an example we apply this formula to the Kronecker algebra $A$ and give an explicit formula to compute the indecomposable decomposition of $M$, which enables us to make a computer program.

\end{abstract}

\section{Introduction}
Throughout this paper $\Bbbk$ is an algebraically closed field, and
all vector spaces, algebras and linear maps are assumed to be finite-dimensional
$\Bbbk$-vector spaces, finite-dimensional $\Bbbk$-algebras and $\Bbbk$-linear maps, respectively.
Furthermore all modules over an algebra considered here are assumed to be
finite-dimensional left modules.

Let $A$ be an algebra,  $\calL$ a complete set  of
representatives of isoclasses of indecomposable $A$-modules.
Then the Krull-Schmidt theorem states the following.
For each $A$-module $M$, there exists a unique map
$\boldsymbol{d}_M \colon \calL \to \bbN_0$ such that
$$
M \iso \Ds_{L \in \calL}L^{(\boldsymbol{d}_M(L))},
$$
which is called an {\em indecomposable decomposition} of $M$.
Therefore, $M \iso N$ if and only if $\boldsymbol{d}_M = \boldsymbol{d}_N$ for all $A$-modules $M$ and $N$,
i.e., the map $\boldsymbol{d}_M$ is a complete invariant of $M$ under isomorphisms.
Note that since $M$ is finite-dimensional, the support
$\supp(\boldsymbol{d}_M):= \{L \in \calL \mid \boldsymbol{d}_M(L) \ne 0\}$ of $\boldsymbol{d}_M$ is a finite set.
We call such a theory a {\em decomposition theory}
that computes the indecomposable decomposition of a module.

Now Krull-Schmidt theorem reduces the description of the module category to
that of the full subcategory of indecomposable modules, for which
the Auslander-Reiten theory was developed since 1970s in representation theory of algebras.
Almost split sequences are the most important notion in the theory that combine
indecomposable modules by irreducible morphisms, by which unknown indecomposables
are computed from known ones.
%
%
In many cases it enabled us to compute the Auslander-Reiten quiver (AR-quiver for short)
of $A$ that is a combinatorial description of the category of modules over $A$,
the vertex set of which can be identified with the list $\calL$. 
It is constructed by gluing all meshes that is a visual form of almost split sequences over $A$.
Thus all information on almost split sequences over $A$ are encoded in the AR-quiver
in a visual way.
Namely, if $0 \to X \to Y \to Z \to 0$ is an almost split sequence,
and $Y = \Ds_{i=1}^n Y_i^{(a_i)}$
\ ($n \ge 1$) is an indecomposable decomposition of $Y$
with $Y_i$ pairwise non-isomorphic and $a_i \ge 1$ for all $i$, then we express it by the quiver
$$
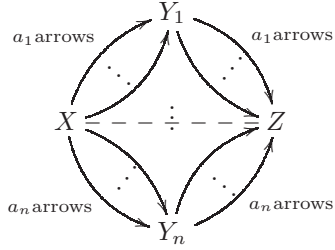
\xymatrix{
&Y_1\\
X &\vdots& Z\\
&Y_n
\ar@/^2ex/^{a_1 \text{arrows}}"2,1";"1,2"
\ar@{}|{\ddots}"2,1";"1,2"
\ar@/_2.3ex/"2,1";"1,2"
\ar@/^2.3ex/^{a_1 \text{arrows}}"1,2";"2,3"
\ar@{}|{\rotatebox{90}{$\ddots$}}"1,2";"2,3"
\ar@/_2ex/"1,2";"2,3"
\ar@{--}"2,1";"2,3"
\ar@/^2ex/"2,1";"3,2"
\ar@{}|{\rotatebox{90}{$\ddots$}}"2,1";"3,2"
\ar@/_2.3ex/_{a_n \text{arrows}}"2,1";"3,2"
\ar@/^2.3ex/"3,2";"2,3"
\ar@{}|{\ddots}"3,2";"2,3"
\ar@/_2ex/_{a_n \text{arrows}}"3,2";"2,3"
}
$$
with a broken line between $X$ and $Z$
(note here that also both $X, Z$ are indecomposable by definition of almost split sequences
).
The correspondence $\ta\colon Z \mapsto X$ is called the {\em AR-translation}.
For example, it has the forms
$$
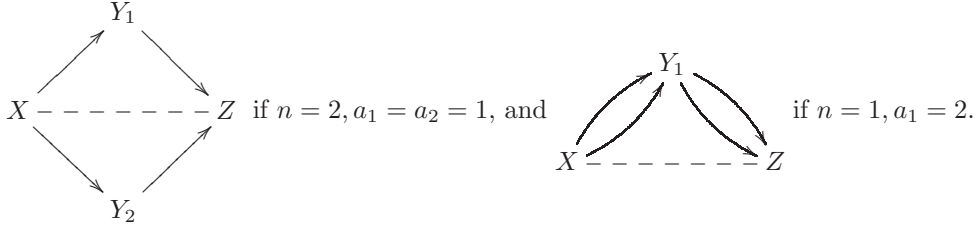
\vcenter{
\xymatrix{
& Y_1\\
X && Z\\
& Y_2
\ar@{--}"2,1";"2,3"
\ar"2,1";"1,2"
\ar"2,1";"3,2"
\ar"1,2";"2,3"
\ar"3,2";"2,3"
}}
\text{ if $n = 2, a_1 = a_2 = 1$, and}
\vcenter{\xymatrix{
&Y_1\\
X && Z
\ar@{--}"2,1";"2,3"
\ar@/^/"2,1";"1,2"
\ar@/_/"2,1";"1,2"
\ar@/^/"1,2";"2,3"
\ar@/_/"1,2";"2,3"
}}
\text{if $n=1, a_1 =2$.}
$$
See \cite{ASS} for details.

The purpose of this paper is to develop a decomposition theory by using
the knowledge of AR-quivers.
Thus in the case that $\calL$ is already computed and all almost split sequences are known,
we aim to compute
\begin{enumerate}
\item[(I)]
$\boldsymbol{d}_M$
and
\item[(II)]
a finite set $S_M$ such that $\supp(\boldsymbol{d}_M) \subseteq S_M \subseteq \calL$
\end{enumerate}
for all $A$-modules $M$.
Note that (II) is needed to give a finite algorithm.
If $A$ is representation-finite (i.e., if the set $\calL$ is finite), then
the problem (II) is trivial because we can take $S_M := \calL$.

In the topological data analysis, to analyse the topological features of a point cloud $C$,
a set of points in $\bbR^u$ for some fixed positive integer $u$,
the persistent homology $M_C$ may be used that encodes topological information on $C$.
Namely, fix a sequence of positive real numbers $r_1 < r_2 < \cdots < r_n$.
For each $i$ a simplicial complex $X(r_i)$ is defined by considering balls of radius $r_i$ with center $c$ ($c \in C$).
This defines a sequence $M_C$
of $k$-th homologies:
$$
H_k(X(r_1)) \to H_k(X(r_2)) \to \cdots \to H_k(X(r_n))
$$
for a fixed dimension $k$,
which encodes the ``lifetime'' of the $k$-dimensional ``hole''.
Now the persistent homology $M_C$ is just a module over
the path algebra $\La_n = \Bbbk Q_n$ of a quiver $Q_n$ of the form
$$
1 \ya{\al_1} 2 \ya{\al_2} \cdots \ya{\al_{n-1}} n
$$
of Dynkin type $A_n$ for some positive integer $n$.
%
%
Therefore to understand the topological features of a point cloud $C$ we can use the knowledge of the
map $\bd_{M_C}$, which is nothing but the ``persistence diagram'' of $C$.
Usually the values of $\bd_{M_C}(L)$ \ ($L \in \calL$) is presented by colors on
$\calL$, and $\calL$ is expressed by a set of lattice points in a triangle.
More precisely, 
the list $\calL$ is given by $\{\mathbb{I}(b, d) \mid 1 \le b \le d \le n\}$
thanks to Gabriel's theorem on representations of Dynkin quivers, where $\mathbb{I}(b, d)$ is given by
$$
0 \ya{} \cdots \ya{} 0\ya{}\Bbbk \ya{1} \Bbbk \ya{1}\cdots \ya{1}\Bbbk \ya{}0
\ya{}\cdots \ya{}0
$$
with $\Bbbk$ starting at the vertex $b$ and stopping at $d$
(this represents the ``lifetime'' $d - b$, ``birth'' $b$ and ``death'' $d$ of a $k$-dimensional hole).
Therefore there exists a 1-1 correspondence between $\calL$ and the set
$\{(b, d) \mid 1 \le b \le d \le n\}$, which is a subset of  $\bbZ^2$ 
forming a triangle (See for instance papers \cite{ZC} and \cite{CS}).
Note that this set of vertices together with horizontal and vertical edges connecting
them can be regarded as the underlying graph of the AR-quiver of $\La_n$.
As an application of our decomposition theory we gave an explicit formula for $\La_n$-modules, i.e., for the persistent diagram (see formula \eqref{eq:An} in Example \ref{exm:An}).  
We refer the reader to \cite{O},  \cite{C}, \cite{CS}  and \cite{EH} for details on  persistent
homology and/or topological data analysis.

To analyse some properties of a set of point clouds,
 e.g., a motion of a point cloud,
persistent homologies were generalized to persistence modules $M$,
which turn out to be modules over an algebra of the form $\La_m \otimes_\Bbbk \La_n$, 
where we allow any orientation of $Q_m$ and $Q_n$.
Namely, their underlying graphs have
the form
$$
\xymatrix{1 \ar@{-}[r] & 2 \ar@{-}[r]& \cdots \ar@{-}[r] & l}
$$
of type $A_l$ for $l = m, n$.
Also in this case we need to compute the persistence diagram $\bd_M$
to investigate the set of point clouds.  It was done in \cite{EH} for the case $(m, n) = (2, 3)$.
Our argument here can be applied to have a decomposition theory for persistence modules.

\begin{example}
\label{exm:Jordan}
The decomposition theory for polynomial algebras in one variable $A=\Bbbk[x]$ is already well known.
A finite-dimensional $A$-module is a pair $(V, f)$ of a finite-dimensional
$\Bbbk$-vector space  $V$ and an endomorphism $f$ of $V$,
and by fixing a basis of $V$ we may regard $V = \Bbbk^d$ for $d:= \dim V$
and $f$ as a square matrix $M$ of size $d$.  In this way we identify $(V, f)$ with $M$.
In this case we may have $\mathcal{L}= \{J_i (\lambda)\ |\ i\geq 1, \la \in \Bbbk \}$, where $J_i (\la)$ is the Jordan cell of size $i \geq 1$ with eigenvalue $\la \in \Bbbk$.
Let $\La$ be the set of all distinct eigenvalues of $M$ and set
$M_{\la}= M-\la E_d$ for $\la\in \La$.
Then the following is well known.
\begin{theorem}The problems $(${\rm I}$)$ and $(${\rm II}$)$ are solved as follows.
\text{}
\begin{description}
\item[A solution to $(${\rm I}$)$]
Let $i \in \bbN$ and $\la \in \La$.  Then
\begin{equation} \label{Jordan-formula}
\boldsymbol{d}_M (J_i (\la))=
\begin{cases}
d+\rank M_{\la}^2 -2 \rank M_{\la}&(i=1)\\
\rank M_{\la}^{i+1} + \rank M_{\la}^{i-1} -2 \rank M_{\la}^{i}  &(i\geq 2)
\end{cases}
\end{equation}
$($Note that by setting $M_\la^0$ to be the identity matrix of size $d$,
the first equality has the same form as the second.$)$
\item[A solution to $(${\rm II}$)$]
$
S_M = \{J_i (\la)\ |\ i\leq d, \la\in \La \}.
$
\end{description}
\end{theorem}
\end{example}

In this paper, 
we will solve the problem (I) in the decomposition theory for any finite-dimensional algebra $A$.
This turns out to be an extension of the result for $A=\Bbbk[x]$ above.
In particular, for the Kronecker algebra $A=\Bbbk Q$ with $Q = (\xymatrix@1{1 \ar@/^/[r]^\al \ar@/_/[r]_\be &2})$, 
we will give an explicit formula for the problem (I) and a solution to the problem (II). 

After submitting the paper we are pointed out by Emerson Escolar and the referee
that there was already a similar investigation \cite{DM} by Dowbor and Mr{\'o}z in the literature, which we did not know before.  Thus this work was done independently.
We here list some relationships between their results and ours.
\begin{enumerate}
\item[(i)]
They also have the same statement as Theorem \ref{main-formula} and
 its dual version,
namely a solution to (I).
Their proof is similar to the first version of ours using a ``Cartan matrix'' of the module
category of an algebra $A$ and an AR-matrix of $A$ as its inverse,
but the proof presented here does not use them and is much simplified by using
the minimal projective resolutions of simple functors that are given by
almost split sequences and sink maps into indecomposable injective modules
(Proposition \ref{prp:proj-res}).
They also used this as a basic fact but not efficiently, namely other considerations
were superfluous for the proof.
\item[(ii)]
Our Theorem \ref{thm:pd} gives an explicit way of computation of the map $\bd_M$
for a module $M$ by using ranks of matrices constructed by the structure matrices of $M$,
while they did not give such formulas explicitly.
\item[(iii)]
Let $M := (\xymatrix@C=2em{\Bbbk^{d_1} \ar@/^/[r]^{M(\al)} \ar@/_/[r]_{M(\be)} &\Bbbk^{d_2}})$ be a representation of the Kronecker algebra,
where $M(\al), M(\be)$ are $d_2 \times d_1$ matrices.
To solve the problem (II) we first 
compute a decomposition
\begin{equation}\label{eq:decomp-M}
M = P_M \ds R'_M \ds R(\infty)_M \ds I_M
\end{equation}
 of the module $M$ into
the preprojective part $P_M$,
the preinjective part $I_M$,
the regular part $R(\infty)_M$ with parameter $\infty$, 
and the regular part $R'_M$ without parameter $\infty$
by using traces and a reject as follows:
\begin{equation}\label{eq:dec-four}
R_M\ds I_M = \rej_M(P_{d_2}), I_M = \tr_{R_M\ds I_M}(I_{d_1}), \text{ and } 
R(\infty)_M=\tr_{R_M}(R_d(\infty)),
\end{equation}
where $R_M:= R'_M \ds R(\infty)_M$ is the regular part of $M$.
By applying  Theorem \ref{thm:pd} to each of 
$P_M, R'_M ,  R(\infty)_M, \text{ and } I_M$ we obtain
the indecomposable decomposition of the module $M$.
Note that each of these modules have the smaller dimensions than that of $M$,
which reduces the complexity of the computation.

On the other hand Dowbor and Mr{\'o}z uses Proposition 4.4 of \cite{DM} stating that the regular indecomposable module
$R_n(\la)$ appears as a direct summand of $M$ if the following hold:
\begin{equation}\label{eq:DM}
\la\text{ is a common root of all }(d_2 - \sum_{i=1}^{n_P}s_i)\text{-minors of the matrix }
M(\al) - x M(\be)
\end{equation}

regarded as polynomials in $\Bbbk[x]$, where
$P_M = P_1^{s_1} \ds \cdots \ds P_{n_P}^{s_{n_P}}$
and $x$ is a variable
(see Theorem \ref{thm:Kronecker} below for the definition of
$P_i, R_n(\la)$, $i = 1, \dots, s_{n_P}$).

The main difference between ours and theirs is in \eqref{eq:dec-four} and \eqref{eq:DM}.
To check \eqref{eq:DM} we need ${d_1 \choose s} \cdot {d_2 \choose s}$ calculations of determinants, where $s = d_2 - \sum_{i=1}^{n_P}s_i$.
Their complexity seems to be much greater than that of \eqref{eq:dec-four} because
the latter uses traces and rejects, which needs time but is computed only three times.
Note also that the decomposition \eqref{eq:decomp-M} given by \eqref{eq:dec-four}
has an important meaning in its own right.
For instance Iwata--Shimizu \cite{IS} was interested in each of the preprojective part $P_M$,
the preinjective part $I_M$ and the regular part $R(\infty)_M$ with parameter $\infty$
and in that case it is needed to take out those parts separately,
which is made possible by our method.
\item[(iv)]
They investigated also the cases of general $\tilde{A}$-quivers and representation-finite
string algebras.
\end{enumerate}

\section{Preliminaries}
Let $m, n$ be non-negative integers.
Then we denote by $\Mat{m,n}$ the vector space of $m\times n$ matrices over $\Bbbk$,
and by $E_n$ the identity matrix of size $n$ (for $n \ge 1$).
By the isomorphism $\Mat{m,n} \to \Hom_\Bbbk(\Bbbk^n, \Bbbk^m)$ sending each
$M \in \Mat{m,n}$ to the linear map given by the left multiplication by $M$
we identify $\Mat{m,n}$ with $\Hom_\Bbbk(\Bbbk^n, \Bbbk^m)$, and regard each $M \in \Mat{m,n}$ as the corresponding linear map
$\Bbbk^n \to \Bbbk^m$.  If $m$ or $n$ is zero, we denote the matrices corresponding to the 
zero maps $\Bbbk^n \to \Bbbk^m$ by $\J_{m,n}$, respectively and call them {\em empty matrices}.

The Kronecker algebra $A$ is a path algebra of the quiver
$Q = (\xymatrix@1{1 \ar@/^/[r]^\al \ar@/_/[r]_\be &2})$, and the category
$\mod A$ of finite-dimensional $A$-modules is equivalent to the category $\rep Q$ of finite-dimensional representations of $Q$ over $\Bbbk$.
We usually identify these categories.
Recall that a representation $M$ of $Q$ is a diagram
$\xymatrix@1{M(1)\ar@/^/[r]^{M(\al)} \ar@/_/[r]_{M(\be)} &M(2)}$
of vector spaces and linear maps, and the {\em dimension vector} of $M$
is defined to be the pair $\udim M:= (\dim M(1), \dim M(2))$.
When $\udim M = (d_1, d_2)$,
without loss of generality we may set $M(i) = \Bbbk^{d_i}$ for $i =1, 2$ and 
$M(\al), M(\be) \in \Mat{d_2, d_1}$.
We denote $M$ by the pair of matrices $(M(\al), M(\be))$.

We here list well known facts on the Kronecker algebra
(see Ringel \cite[3.2]{R} for instance).

\begin{theorem}\label{thm:Kronecker}
For the Kronecker algebra $A$ the following statements hold.

\begin{enumerate}
 \item[$(1)$]
The list $\calL$ of indecomposables is given as follows.

\def\boldzero{\mathbf{0}}
Preprojective indecomposables:
$\calP:=\left\{\left.P_n:= \left(\bmat{E_{n-1}\\{}^t\boldzero}, \bmat{{}^t\boldzero\\E_{n-1}}\right)\right|  n \ge 1\right\}$,

Preinjective indecomposables:
$\calI:=\left\{\left.I_n:= ([E_{n-1}, \boldzero], [\boldzero, E_{n-1}])\, \right| n \ge 1\right\}$,

Regular indecomposables:
$$\calR:=\{R_n(\la):= (E_n, J_n(\la)),
R_n(\infty):= (J_n(0), E_n)\mid n \ge 1, \la \in \Bbbk\},$$
where  $\boldzero$ is the $(n -1) \times 1$
matrix with all entries 0.
Note that
$$
\udim P_n = (n-1, n), \udim I_n = (n, n-1), \udim R_n(\la) = (n,n)
$$
for all $n \in \bbN$ and $\la \in \bbP^1(\Bbbk) = \Bbbk \cup \{\infty\}$.
\item[$(2)$]
The Auslander-Reiten quiver $($AR-quiver for short$)$ of $A$ has the following form:

$$
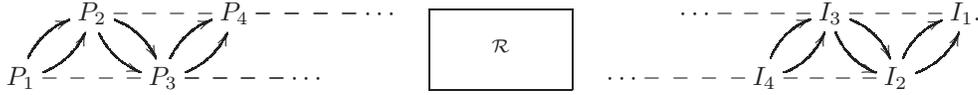
\xymatrix@R=10pt@C=10pt{
&P_2 &&P_4&&\cdots && && && \cdots && I_3 && I_1.\\
P_1&&P_3&&\cdots && && && \cdots&&I_4  && I_2
\ar@/^/"2,1";"1,2"\ar@/_/"2,1";"1,2" \ar@{--}"2,1";"2,3"
\ar@/^/"1,2";"2,3"\ar@/_/"1,2";"2,3" \ar@{--}"1,2";"1,4"
\ar@/^/"2,3";"1,4"\ar@/_/"2,3";"1,4" \ar@{--}"2,3";"2,5"
\ar@{--}"1,4";"1,6"
\ar@/^/"1,14";"2,15"\ar@/_/"1,14";"2,15" \ar@{--}"1,14";"1,16"
\ar@/^/"2,13";"1,14"\ar@/_/"2,13";"1,14" \ar@{--}"2,13";"2,15"
\ar@/^/"2,15";"1,16"\ar@/_/"2,15";"1,16" \ar@{--}"2,3";"2,5"
\ar@{--}"1,4";"1,6"
\ar@{--}"2,13";"2,11"
\ar@{--}"1,14";"1,12"
\ar@{}"1,7";"2,10"|{\calR}
\save "1,7"."2,10"*[F]\frm{} \restore}
$$
In the above the rectangle part $\calR$ is given as the disjoint union of
a family $(\calR_\la)_{\la \in \bbP^1(\Bbbk)}$ of ``homogeneous tubes'' $\calR_\la$
that has the form $$
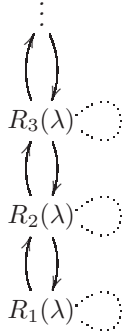
\xymatrix@R=20pt@C=-8pt
{
\vdots\\
R_3(\la)&\\
R_2(\la)&\\
R_1(\la)&
\ar@/^/"4,1";"3,1"\ar@/^/"3,1";"4,1"
\ar@/^/"3,1";"2,1"\ar@/^/"2,1";"3,1"
\ar@/^/"2,1";"1,1"\ar@/^/"1,1";"2,1"
\ar@(dr,ur)@{..}"4,2";"4,2"
\ar@(dr,ur)@{..}"3,2";"3,2"
\ar@(dr,ur)@{..}"2,2";"2,2"
}
$$
where dotted loops mean that for all $n \in \bbN$
the Auslander-Reiten translation $\ta$ sends $R_n(\la)$ to itself:
$\ta R_n(\la) = R_n(\la)$.
\item[$(3)$]
Let $X, Y \in \calL$.  If $\Hom_A(X, Y) \ne 0$, then $X$ is ``on the left'' of $Y$, i.e.,
one of the following occurs:
\begin{itemize}
\item[\rm (i)]
$X \iso P_m, Y \iso P_n$ with $m\le n$,
\item[\rm (ii)]
$X \in \calP, Y \in \calR \cup \calI$,
\item[\rm (iii)]
$X \iso R_m(\la), Y \iso R_n(\mu)$ with $\la = \mu$,
\item[\rm (iv)]
$X \in \calR, Y \in \calI$, or
\item[\rm (v)]
$X \iso I_m, Y \iso I_n$ with $m \ge n$.
\end{itemize}
\end{enumerate}
\end{theorem}

\begin{remark}\label{rmk:order}
(1) Let $m, n \in \bbZ$ with $m \le n$.
Then we note that there exists a monomorphism $P_m \to P_n$ and
an epimorphism $I_n \to I_m$.

(2) Now for $(a_1,a_2), (b_1,b_2) \in \bbZ^2$ we define $(a_1,a_2) \le (b_1,b_2)$ if and only if $a_i \le b_i$ 
for $i= 1 \text{ and } 2$.
Then if there exists a monomorphism $T \to U$ (or an epimorphism $U \to T$) in $\mod A$,
we have $\udim T \le \udim U$.
\end{remark}

\section{Simple functors: a solution to the problem (I) in general}

In this section we give a solution to the problem (I) by using Auslander-Reiten theory for an arbitrary algebra $A$.

\begin{definition}
For an indecomposable $A$-module $L$ we set
$$
\calS_L := \Hom_A (L,\blank)/\rad \Hom_A(L,\blank) : \mod A \to \mod \Bbbk.
$$
It is well-known that $\calS_L$ is a simple functor.
\end{definition}

\begin{lemma}\label{simple-dim}
Let $M$ be an $A$-module.
Then for any indecomposable $A$-module $L$ we have
$$
\boldsymbol{d}_M(L)=\dim \calS_L (M).
$$
\end{lemma}

\begin{proof}
Since $L$ is indecomposable, $\End_A(L)$ is a local algebra.
Therefore
$\calS_L(L) = \End_A(L)/\rad(\End_A(L))$ is a finite-dimensional skew field over
the algebraically closed field $\Bbbk$, and hence $\calS_L(L) \iso \Bbbk$.
If $X \not\iso L$, then $\End_A(L)=\rad(\End_A(L))$, and $\calS_L(X) = 0$.
Thus
$$
\calS_L (X) \iso
\left\{
\begin{array}{ll}
\Bbbk &  \text{if}\ X \iso L \\
0 & \text{if}\ X \niso L
\end{array}
\right. 
$$
for all indecomposable $A$-modules $X$.
Therefore, the indecomposable decomposition
$$
M \iso \Ds_{L \in \calL}L^{(\boldsymbol{d}_M(L))}
$$
of $M$ gives us
$$
\calS_L (M) \iso \Bbbk^{(\boldsymbol{d}_M(L))},
$$
which shows the assertion.
\end{proof}

Recall the following fundamental statement in the Auslander-Reiten theory
(see Auslander-Reiten \cite{AR} or Assem-Simson-Skowro\'nski \cite[IV, 6.11.]{ASS}):

\begin{proposition}\label{prp:proj-res}
Let $L$ be an indecomposable $A$-module.
When $L$ is non-injective, let
$
0 \to L \ya{\ f\ } \DS_{X\in J_L} X^{(a(X))} \ya{\ g\ } \tau^{-1}L \to 0
$
be an almost split sequence starting at $L$ with $J_L \subseteq \calL$ and $a(X) \ge 1$
$(X \in J_L)$.
When $L$ is injective, let $f\colon L \to L/\soc L =  \DS_{X\in J_L} X^{(a(X))}$
be the canonical epimorphism $($note that $J_L = \emptyset$ if $L$ is simple injective$)$.
Then the simple functor $\calS_L$ has a minimal projective resolution
$$
0 \to \Hom_A(\tau^{-1}L,\blank) \ya{\Hom_A(g,\blank)}
\DS_{X\in J_L}\Hom_A( X, \blank)^{(a(X))} \ya{\Hom_A(f,\blank)}
\Hom_A(L,\blank) \ya{\text{can}} \calS_L \to 0,
$$
where $g=0$ and $\ta\inv L = 0$ if $L$ is injective.
\end{proposition}

Proposition \ref{prp:proj-res} together with Lemma \ref{simple-dim} readily gives us
the following.

\begin{theorem}\label{main-formula}
Let $M$ be an $A$-module.  Then for any indecomposable $A$-module $L$ we have
$$
\boldsymbol{d}_M(L)= \dim \Hom_A (L, M) - \Sm_{X\in J_L} a(X) \dim\Hom_A( X, M)+ \dim \Hom_A (\tau^{-1}L , M).
$$
\end{theorem}

\begin{remark} \label{calc-Homdim}
When an algebra $A$ is of the form $\Bbbk Q / I$ for some quiver $Q$
and some ideal $I$ of  $\Bbbk Q$, 
it is possible to compute $\dim \Hom_A (H,M)$ for every $H,M\in \mod A$ by using
the $\rank$ of a suitable matrix as follows,
and thus $\boldsymbol{d}_M(L)$ in Theorem \ref{main-formula} is computable. 
First regard $A$-modules $H$ and $M$ as representations $(H(i),H(\al))_{i\in Q_0, \al\in Q_1}$ and $(M(i),M(\al))_{i\in Q_0, \al\in Q_1}$ of $Q$, respectively.
Then by definition we have
\begin{equation}\label{eq:Hom}
\begin{split}
\Hom_A (H,M)= \{ (f_i)_{i\in Q_0} \in \prod_{i\in Q_0} \Hom_{\Bbbk} (H(i),M(i))\ \mid
M(\al) f_i = f_j H(\al) , 
\forall \al :i \to j\ \text{in}\ Q_1
\}.
\end{split}
\end{equation}
Therefore
$$
\Hom_A (H,M) \iso \{  \bdx \in \Bbbk ^ N \ |\ B \bdx =0 \},
$$
where $N:= \Sm_{i \in Q_0} \dim H(i) \dim M(i)$ and $B$ is
a
$c\times N$-matrix given as the coefficient matrix of the homogeneous system of linear equations $M(\al) f_i -  f_j H(\al) = 0$ for $f_i$, where
$c:= \displaystyle\sum_{\al: i \to j \text{ in } Q_1}\!\!\!\!\dim M(j) \dim H(i)$. 
Hence we obtain the equality:
$$
\dim \Hom_A (H,M) = N - \rank B.
$$
\end{remark}

\begin{example}\label{exm:Jordan-PD}
Let $A:= \Bbbk[x]$ be the polynomial algebra in one variable.
Although it is an infinite-dimensional algebra, the category
$\mod A$ of finite-dimensional $A$-modules is well understood
because $\Bbbk[x]$ is a principal ideal domain, and we can apply Auslander-Reiten theory
to $\mod A$.
It is easy to give all almost split sequences over $\Bbbk[x]$.
Namely, they are given as follows:
\begin{equation} \label{eq:ass-Jordan}
\begin{gathered}
0 \to J_1(\la) \to J_2(\la) \to J_1(\la) \to 0,\\
0 \to J_i(\la) \to J_{i-1}(\la) \ds J_{i+1}(\la) \to J_i(\la) \to 0
\end{gathered}
\end{equation}
for all $i \ge 2$ and $\la \in \Bbbk$.
This is verified by the similar argument used in the Nakayama algebra case
(cf.\ \cite[4.1 Theorem]{ASS}).
The reader may notice a similarity between \eqref{Jordan-formula} and \eqref{eq:ass-Jordan},
which will become clear now.
Let $M=(\Bbbk^d, M)$ be an $A$-module.
Then we have
\begin{equation}\label{eq:Hom-Jordan}
\dim \Hom_A (J_{i}(\la), M)=d-\rank M^i_{\la},
\end{equation}
which together with Theorem \ref{main-formula} and the formula \eqref{eq:ass-Jordan} yields the formula \eqref{Jordan-formula}. 

Indeed, let $X \in \Mat{d,i}$, and
put $X_j$ to be the $j$-th column of $X$ ($j=1,\dots, i$).
Then by \eqref{eq:Hom} $X \in \Hom_A (J_{i}(\la), M)$ iff
$MX=XJ_i(\la)=X(\la E_i+ J_i(0))=\la X+XJ_i(0)$
iff $M_\la X=XJ_i(0)$
iff
$M_\la (X_1,\dots, X_i) = (0,X_1,\dots, X_{i-1})$ iff $M_\la$ maps $X_j$'s as follows 
$$
X_i \mapsto X_{i-1}\mapsto \cdots \mapsto X_1\mapsto 0.
$$
Hence the correspondence $X \mapsto X_i$ yields the isomorphism
(the inverse is given by the correspondence $v \mapsto [M_\la^{i-1}v,\dots, M_\la v, v]$)
$$
\Hom_A (J_{i}(\la), M)\iso \{v \in\Bbbk^d \mid M_\la^i v = 0\}=\Ker M_\la^i,
$$
which shows the equality \eqref{eq:Hom-Jordan}.
\end{example}

\begin{example}\label{exm:An}
Fix a positive integer $n$ and set $A:=\La_n =\Bbbk Q_n$, where $Q_n$ is a Dynkin quiver $1 \ya{\al_1} 2 \ya{\al_2} \cdots \ya{\al_{n-1}} n$
 of type $A_n$. Decomposition Theory for modules over this algebra has important applications in the topological data analysis (See the introduction). 
Let $M:=(M_i)_{i=1}^{n}:=
(\Bbbk^{a_1} \ya{M_1}\Bbbk^{a_2} \ya{M_2}\cdots \ya{M_{n-1}}\Bbbk^{a_n}$) be a representation of $Q$ (i.e. an $A$-module).
Then the morphism space $\Hom_A(\mathbb{I}(b,d),M)$ is the set of sequences $(f_i:\mathbb{I}(b,d)(i)\ya{}\Bbbk^{a_i})_{i=1}^{n}$ 
that make the following diagram commutative:
$$
\xymatrix@R=20pt@C=16pt{
0 & \cdots & 0 & \Bbbk & \cdots & \Bbbk & 0 & \cdots & 0\\
\Bbbk^{a_1} & \cdots & \Bbbk^{a_{b-1}} & \Bbbk^{a_{b}} & \cdots & \Bbbk^{a_{d}} 
& \Bbbk^{a_{d+1}} & \cdots & \Bbbk^{a_n},
\ar^{0}"1,1";"1,2" \ar^{0}"1,2";"1,3" \ar^{0}"1,3";"1,4" \ar^{1}"1,4";"1,5"
\ar^{1}"1,5";"1,6" \ar^{0}"1,6";"1,7" \ar^{0}"1,7";"1,8" \ar^{0}"1,8";"1,9"
\ar_{0}"1,1";"2,1" \ar_{0}"1,3";"2,3" \ar_{f_b}"1,4";"2,4"
\ar_{f_d}"1,6";"2,6" \ar_{0}"1,7";"2,7" \ar_{0}"1,9";"2,9"
\ar_{M_1}"2,1";"2,2" \ar_{M_{b-2}}"2,2";"2,3" \ar_{M_{b-1}}"2,3";"2,4" \ar_{M_b}"2,4";"2,5"
\ar_{M_{d-1}}"2,5";"2,6" \ar_{M_d}"2,6";"2,7" \ar_{M_{d+1}}"2,7";"2,8" \ar_{M_{n-1}}"2,8";"2,9"
\ar@{}|<<<<<<{\circlearrowright}"1,1";"2,2"
\ar@{}|<<<<<<{\circlearrowright}"1,2";"2,3"
\ar@{}|<<<<<<<{\circlearrowright}"1,3";"2,4"
\ar@{}|<<<<<<{\circlearrowright}"1,4";"2,5"
\ar@{}|<<<<<<{\circlearrowright}"1,5";"2,6"
\ar@{}|<<<<<<<{\circlearrowright}"1,6";"2,7"
\ar@{}|<<<<<<<{\circlearrowright}"1,7";"2,8"
\ar@{}|<<<<<<{\circlearrowright}"1,8";"2,9"
}
$$
where $
\mathbb{I}(b,d)(i):=
\begin{cases}
\Bbbk & (b\leq i \leq d)\\
0 & (\text{otherwise}) 
\end{cases}$ as in the introduction. In particular, if $d=n$ (namely $\mathbb{I}(b,d)$ is projective), then
$$
\Hom_A(\mathbb{I}(b,n),M)
\iso \{(f_i)_{i=b}^{n}\mid M_{b}f_{b}=f_{b+1},\dots,M_{n-1}f_{n-1}=f_{n} \}
\iso \Bbbk^{a_b},
$$
and if $d\leq n-1$, then
$$
\begin{aligned}
\Hom_A(\mathbb{I}(b,d),M)
&\iso \{(f_i)_{i=b}^{d}\mid M_{b}f_{b}=f_{b+1},\dots,M_{d-1}f_{d-1}=f_{d},M_{d}f_{d}=0 \}\\
&\iso \{ f_b \in \Bbbk^{a_b} \mid M_{d}M_{d-1}\cdots M_b f_{b}=0 \}.
\end{aligned}
$$
Hence we obtain 
\begin{equation}\label{eq:An-Hom}
\dim \Hom_A (\mathbb{I}(b,d),M)=
a_b-\rank (M_{d}M_{d-1}\cdots M_b),
\end{equation}
where we set $M_n:=0$.
Since the AR-quiver $\Ga_A$ of $A$ is of the following form: 
$$
{\fontsize{5pt}{5pt}\selectfont
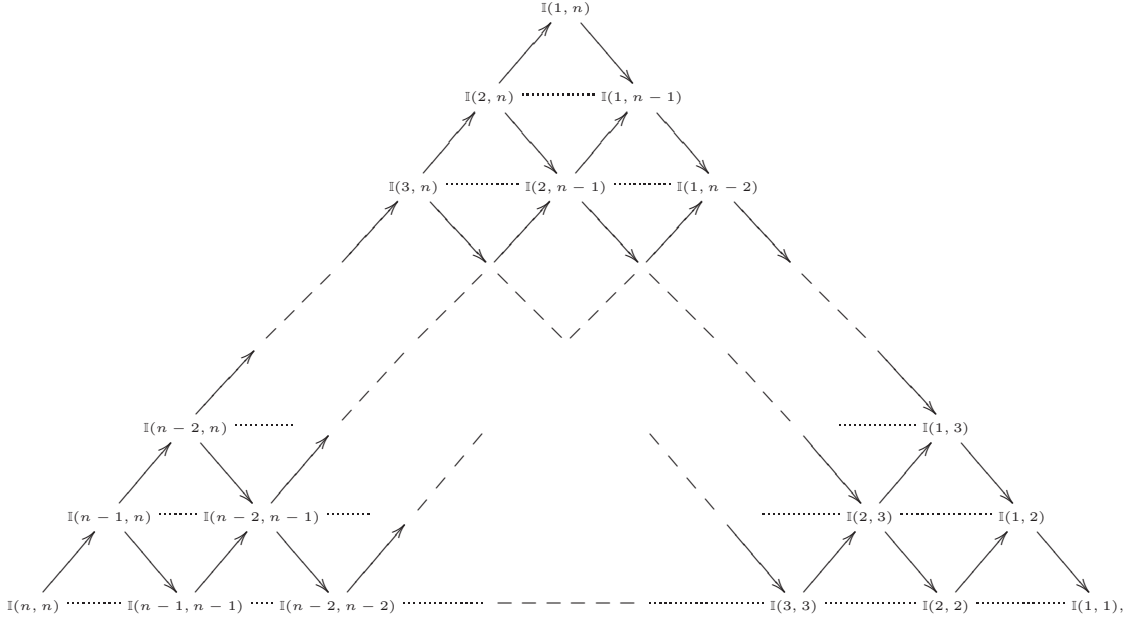
\xymatrix@R=8mm@C=-12mm{
  &  &  &  &  &  &  & \mathbb{I}(1,n) &  &  &  &  &  &  &  \\
  &  &  &  &  &  & \mathbb{I}(2,n) &  & \mathbb{I}(1,n-1) &  &  &  &  &  &  \\
  &  &  &  &  & \mathbb{I}(3,n) &  & \mathbb{I}(2,n-1) &  & \mathbb{I}(1,n-2) &  &  &  &  &  \\
  &  &  &  &  &  &  &  &  &  &  &  &  &  &  \\
  &  &  &  &  & \hspace{10mm} &  & \hspace{10mm} &  & \hspace{10mm} &  &  &  &  &  \\
  &  & \mathbb{I}(n-2,n) &  & \hspace{10mm} &  &  &  &  &  & \hspace{10mm} &  & \mathbb{I}(1,3) &  &  \\
  & \mathbb{I}(n-1,n) &  & \mathbb{I}(n-2,n-1) &  & \hspace{10mm} &  &  &  & \hspace{10mm} &  & \mathbb{I}(2,3) &  & \mathbb{I}(1,2) &  \\
\mathbb{I}(n,n) &  & \mathbb{I}(n-1,n-1) &  & \mathbb{I}(n-2,n-2) &  &  &  &  &  & \mathbb{I}(3,3) &  & \mathbb{I}(2,2) &  & \mathbb{I}(1,1), \\
\hspace{20mm} & \hspace{20mm} & \hspace{20mm} & \hspace{20mm} & \hspace{20mm} & \hspace{20mm} & \hspace{20mm} 
& \hspace{20mm} & \hspace{20mm} & \hspace{20mm} & \hspace{20mm} & \hspace{20mm}  & \hspace{20mm} & \hspace{20mm} & \hspace{20mm} 
\ar"8,1";"7,2" \ar"7,2";"6,3" \ar"6,3";"5,4" \ar@{--}"5,4";"4,5" \ar"4,5";"3,6" \ar"3,6";"2,7" \ar"2,7";"1,8" 
\ar"1,8";"2,9" \ar"2,9";"3,10" \ar"3,10";"4,11" \ar@{--}"4,11";"5,12" \ar"5,12";"6,13" \ar"6,13";"7,14" \ar"7,14";"8,15" 
\ar"7,2";"8,3" \ar"8,3";"7,4" \ar"6,3";"7,4" \ar"7,4";"8,5" 
\ar"2,7";"3,8" \ar"3,8";"2,9" 
\ar"8,11";"7,12" \ar"7,12";"6,13" \ar"7,12";"8,13" \ar"8,13";"7,14" 
\ar"7,4";"6,5" \ar@{--}"6,5";"5,6" \ar"8,5";"7,6" \ar@{--}"7,6";"6,7"
\ar"6,11";"7,12" \ar@{--}"5,10";"6,11" \ar"7,10";"8,11" \ar@{--}"6,9";"7,10"
\ar"3,6";"4,7" \ar"4,7";"3,8" \ar"3,8";"4,9" \ar"4,9";"3,10"
\ar@{--}"5,6";"4,7" \ar@{--}"4,7";"5,8" \ar@{--}"5,8";"4,9" \ar@{--}"4,9";"5,10"
\ar@{.}"2,7";"2,9"
\ar@{.}"3,6";"3,8" \ar@{.}"3,8";"3,10"
\ar@{.}"6,3";"6,5" \ar@{.}"6,11";"6,13"
\ar@{.}"7,2";"7,4" \ar@{.}"7,4";"7,6" \ar@{.}"7,10";"7,12" \ar@{.}"7,12";"7,14"
\ar@{.}"8,1";"8,3" \ar@{.}"8,3";"8,5" \ar@{.}"8,5";"8,7" \ar@{--}"8,7";"8,9" \ar@{.}"8,9";"8,11" \ar@{.}"8,11";"8,13" \ar@{.}"8,13";"8,15"
}}
$$
\\[-10mm]
the formula \eqref{eq:An-Hom} and Theorem \ref{main-formula} give us the formula
\begin{equation}\label{eq:An}
\bd_M(\mathbb{I}(b,d))=
\begin{cases}
\begin{aligned}
\rank (M_d\cdots M_{b-1}) +& \rank(M_{d-1}\cdots M_{b})\\
-&\{\rank(M_{d-1}\cdots M_{b-1}) + \rank(M_d\cdots M_b)\}
\end{aligned}
 &  (b < d)\\
\rank(M_d M_{d-1}) + a_b - \{\rank(M_{d-1}) + \rank(M_d)\} &  (b = d),
\end{cases}
\end{equation}
where we set $M_0:=0$ and $M_n:=0$.
Therefore if we set
$$
R(b,d):=
\begin{cases}
\rank(M_d\cdots M_b)-\rank(M_{d-1}\cdots M_b) & (b<d) \\
\rank(M_b)-a_b & (b=d),
\end{cases}
$$
then we have
$$
\boldsymbol{d}_M(\mathbb{I}(b,d))=R(b-1,d)-R(b,d)
$$
for each $(b,d) \in \{(i,j)\in \bbZ^2 \mid 1\leq i \leq j \leq n\}$.
\end{example}
\section{Solution to the problem (I) for the Kronecker algebra}

Throughout the rest of this paper $A$ is the Kronecker algebra.
To apply Theorem \ref{main-formula}
we compute the dimensions of
the spaces $\Hom_A(L, M)$ for all $L \in \calL$ and $M \in \mod A$
following Remark \ref{calc-Homdim}.

\begin{definition}\label{dfn:rank-mat}
Let $M$ be an $A$-module.
We first define the following matrices with $n \ge 1$, $\la \in \Bbbk$
(note that $P_1(M) = \J_{0,1}$ is an empty matrix).

$$
P_n(M):=
\left.\left[ 
\rule{0pt}{45pt}
\smash{
\overbrace{
    \begin{array}{cccccc}
           M(\be)&M(\al)&0&0&\cdots &0\\
           0&M(\be)&M(\al)&0&\ddots &\vdots\\
           0&0&M (\be)&M(\al)  &\ddots & 0\\
           \vdots&\vdots&\ddots&\ddots& \ddots&0 \\
           0&0&\cdots&0 &M(\be) & M(\al)
    \end{array}
    }^{\text{$n$ blocks}}
    }
\right]\right\}
\,\text{\scriptsize $n-1$ blocks},
$$
\\
$$
I_n(M):=
\left.\left[ 
\rule{0pt}{55pt}
\smash{\overbrace{
    \begin{array}{cccccc}
           M(\be)&0&0&\cdots &0\\
           M(\al)&M(\be)&0&\ddots &\vdots\\
           0&M(\al)&M (\be)&\ddots &0\\
           0&0&M(\al)&\ddots &0 \\
           \vdots&\vdots&\ddots&\ddots& M(\be)\\
           0&0&\cdots&0 &M(\al) \\
    \end{array}
    }^{\text{$n$ blocks}}}
\right]\right\}
\,\text{\scriptsize $n+1$ blocks},
$$
\\
$$
R_n(\la,M):=
\left.\left[ 
\rule{0pt}{45pt}
\smash{\overbrace{
    \begin{array}{cccccc}
           M_\la (\al,\be)&0&0&\cdots &0\\
           M(\al)& M_\la (\al,\be)&0&\ddots &\vdots\\
           0&M(\al)& M_\la (\al,\be)&\ddots &0\\
           \vdots&\ddots&\ddots&\ddots &0 \\
           0&\cdots&0&M(\al)&  M_\la (\al,\be)\\
    \end{array}
    }^{\text{$n$ blocks}}}
\right]\right\}
\,\text{\scriptsize $n$ blocks}, and
$$
\\
$$
R_n(\infty,M):=
\left.\left[ 
\rule{0pt}{45pt}
\smash{\overbrace{
    \begin{array}{cccccc}
           M(\al)&0&0&\cdots &0\\
           -M(\be)& M(\al)&0&\ddots &\vdots\\
           0&-M(\be)& M(\al)&\ddots &0\\
           \vdots&\ddots&\ddots&\ddots &0 \\
           0&\cdots&0&-M(\be)&  M(\al)\\
    \end{array}
    }^{\text{$n$ blocks}}}
\right]\right\}
\,\text{\scriptsize $n$ blocks},
$$
where we put $M_\la (\al,\be):=\la M(\al) -M(\be)$, and we define
the following numbers.
$$
\begin{aligned}
p_1(M)&:= 0, p_n(M):= \rank P_n (M)\ (n\ge 2), \\
i_0(M)&:= 0, i_n(M):= \rank I_n (M)\ (n\ge 1), \\
r_n(\la,M)&:= \rank R_n (\la,M)\ (n\ge 1, \la \in \bbP^1(\Bbbk)).
\end{aligned}
$$
\end{definition}

Using the data above we can compute the dimensions of Hom spaces
$\Hom_A(L, M)$ with $L$ indecomposable as follows.

\begin{proposition}\label{prp:hom-dim}
Let $M$ be an $A$-module.
Then we have the following formulas:
$$
\begin{aligned}
\dim \Hom_A (P_n,M)&=
\begin{cases}
(n-1)d_1-p_{n-1}(M) & (n\geq2)  \\
d_2 & (n=1)\\
\end{cases}
\\
\dim \Hom_A (I_n,M)&=nd_1-i_n(M) \quad (n\geq1)  \\
\dim \Hom_A (R_n(\la),M)&=nd_1-r_n(\la,M) \quad (n\geq1,\la\in \bbP^1(\Bbbk))
\end{aligned}
$$
\end{proposition}

\begin{proof}\label{eq:Pn-M}
Assume that $n \ge 2$.
Let $(X, Y) \in \Mat{d_1, n-1}\times \Mat{d_2, n}$,
and put $X_i$ (resp.\ $Y_i$) to be $i$-th column of $X$ ($i=1,\dots, n-1$) 
(resp.\ $Y$ ($i=1,\dots, n$)).
Then by \eqref{eq:Hom}
$(X,Y) \in \Hom_A (P_n,M)$ iff 
$$
M(\al)X=Y\bmat{E_{n-1}\\0},\quad M(\be)X=Y\bmat{0 \\E_{n-1}}
$$
iff
$$
\begin{cases}
M(\al)X_1=Y_1, M(\al)X_2=Y_2,\dots, M(\al)X_{n-1}=Y_{n-1}\\
M(\be)X_1=Y_2, M(\be)X_2=Y_3,\dots, M(\be)X_{n-1}=Y_n
\end{cases}
$$
iff
$$
\substack{n-1\text{ blocks}\left\{\rule{0pt}{25pt}\right.\\n-1\text{ blocks}\left\{\rule{0pt}{25pt}\right.}
\left[\rule{0pt}{50pt}
    \begin{array}{c|c}
    \\
  \smash{\overbrace{\Maa}^{n-1\text{ blocks}}}&\smash{\overbrace{\Mab}^{n\text{ blocks}}} \\\\\\[-10pt]
    \hline\\[-10pt]
    \Mba&\Mbb
    \end{array}
\right]
\left[\smat{ X_1\\ X_2\\ \vdots \\X_{n-1}\\Y_1 \\Y_2\\ \vdots \\ Y_n}\right]
=0
$$
Let $B$ be the coefficient matrix of this equation.
Then a direct calculation shows that $B$ is equivalent to $P_{n-1}(M) \ds E_{nd_2}$.
Therefore $\rank B =nd_2+p_{n-1}(M)$, which shows that
$\dim \Hom_A(P_n,M)=(n-1)d_1 + nd_2 -\rank B= (n-1)d_1 - p_{n-1}(M)$,
as desired.
The remaining formulas are proved similarly.
\end{proof}

Propositions \ref{prp:hom-dim} and Theorem \ref{main-formula}
give us a solution to the problem (I) for the Kronecker algebra as follows.

\begin{theorem}\label{thm:pd}
Let $M$ be an $A$-module.  Then we have the following formulas:
$$
\begin{aligned}
\boldsymbol{d}_M(P_n)&=
\begin{cases}
2p_n(M) - p_{n-1}(M) - p_{n+1}(M) &  (n\geq2)  \\
d_2-p_2(M) &  (n=1),
\end{cases}
\\
\boldsymbol{d}_M(I_n)&=
\begin{cases}
2i_{n-1}(M) - i_n(M) - i_{n-2}(M) & (n\geq2)  \\
d_1-i_1(M) & (n=1),
\end{cases}
\\
\boldsymbol{d}_M(R_n(\la))&=
\begin{cases}
r_{n-1}(\la,M) + r_{n+1}(\la,M) - 2r_n(\la,M) & (n\geq2)  \\
r_2(\la,M) - 2r_1(\la,M) & (n=1).
\end{cases}
\end{aligned}
$$
Here we note that $\bd_M(P_1)$ and $\bd_M(I_1)$ have obvious meanings that
$\bd_M(P_1) =\dim \Cok[M(\be)\ M(\al)]$ and
$\bd_M(I_1) = \dim \Ker\!\bmat{M(\be)\\M(\al)}$.
\end{theorem}

\begin{proof}
Note that by Theorem \ref{thm:Kronecker}(2) we know all the almost split sequences for
the Kronecker algebra. Therefore we can apply Theorem \ref{main-formula}.
We first compute $\boldsymbol{d}_M(P_1)$ and $\boldsymbol{d}_M(I_1)$.
Noting that $\dim \Hom_A(P_2, M) = d_1 -p_1(M) = d_1$ the almost split sequence
starting at $P_1$ that is given by the mesh starting at $P_1$ in the AR-quiver shows that
$$
\begin{aligned}
\boldsymbol{d}_M(P_1) &= \dim\Hom_A(P_1, M) -2\dim\Hom_A(P_2, M) + \dim\Hom_A(P_3, M)\\
&= d_2 -2d_1 +2d_1 -p_2(M) = d_2 - p_2(M)  \\
&=d_2 - \rank[M(\be)\  M(\al)] = \dim \Cok [M(\be)\  M(\al)].
\end{aligned}
$$
Now since $I_1$ is simple and injective, we have $I_1/\soc I_1 = 0$ and $\ta\inv I_1 = 0.$
Hence
$$
\begin{aligned}
\boldsymbol{d}_M(I_1) &= \dim\Hom_A(I_1, M) = d_1 -i_1(M) \\
&= d_1 - \rank\!\bmat{M(\be)\\M(\al)} = \dim \Ker\!\bmat{M(\be)\\M(\al)}.
\end{aligned}
$$
Next we compute $\boldsymbol{d}_M(P_n)$ for $n \ge 2$.
$$
\begin{aligned}
\boldsymbol{d}_M(P_n) &= \dim\Hom_A(P_n, M)  -2\dim\Hom_A(P_{n+1}, M) + \dim\Hom_A(P_{n+2}, M)\\
&= (n-1)d_1 -p_{n-1}(M) - 2(nd_1 - p_n(M)) + (n+1)d_1 -p_{n+1}(M)\\
&=2p_n(M)-p_{n-1}(M)
-p_{n+1}(M),\end{aligned}
$$
as desired.
The remaining cases are proved similarly.
\end{proof}

\section{Solution to the problem (II) for the Kronecker algebra}

Let $F\colon \Ds_{L \in \calL}L^{(\boldsymbol{d}_M(L))}\to M$
be an isomorphism.
Then we have
$$
M=P_M \oplus R_M \oplus I_M,
$$
where $P_M,R_M$ and $I_M$ are
the images of $\Ds_{L \in \calP}L^{(\boldsymbol{d}_M(L))},\Ds_{L \in \calR}L^{(\boldsymbol{d}_M(L))}$ and
$\Ds_{L \in \calI}L^{(\boldsymbol{d}_M(L))}$ by $F$, respectively.
To compute $P_M, R_M$ and $I_M$
we here use the trace and reject in a module of a class of modules
(see Anderson--Fuller \cite{AF} for details).
Let $\calU$ be a class of modules in $\mod A$ and $M \in \mod A$.
Recall that the {\em trace} $\tr_M(\calU)$ of $\calU$ in $M$ and the {\em reject}
$\rej_M(\calU)$ of $\calU$ in $M$
are defined by
$$
\begin{aligned}
\tr_M(\calU)&:= \sum\{\im f\mid f \in \Hom_A(U, M) \text{ for some }U \in \calU\}, \text{and}\\
\rej_M(\calU)&:= \bigcap\{\Ker f \mid f \in \Hom_A(M, U) \text{ for some }U \in \calU\}.
\end{aligned}
$$
When $\calU = \{U\}$ is a singleton, we set $\tr_M(U):= \tr_M(\calU)$ and
$\rej_M(U):= \rej_M(\calU)$.
We cite the following from \cite[8.18 Proposition]{AF}.

\begin{lemma}\label{lem:tr-rej}
Let $(M_i)_{i\in I}$ be a family of $A$-modules indexed by a set $I$ and $\calU$
a class of modules in $\mod A$.
Then we have
$$
\tr_{\Ds_{i \in I}M_i}(\calU) = \Ds_{i\in I}\tr_{M_i}(\calU)\quad \text{and}\quad
\rej_{\Ds_{i \in I}M_i}(\calU) = \Ds_{i\in I}\rej_{M_i}(\calU).
$$
\end{lemma}

\begin{proposition}[Calculation of $R_M \oplus I_M$]\label{prp:RM-IM}
If $\{f_1,\dots,f_a\}$ is a basis of\/ $\Hom_A(M,P_{d_2})$, then we have
$$\bigcap^a_{i=1}\Ker f_i = R_M \oplus I_M \quad\text{and hence}\quad P_M \iso M/\left(\bigcap^a_{i=1}\Ker f_i\right).$$
\end{proposition}

\begin{proof}
By assumption it is obvious that $\bigcap^a_{i=1}\Ker f_i= \rej_M(P_{d_2})$.
Therefore, it is enough to show that
\begin{equation}\label{eq:RM-IM}
\rej_M(P_{d_2}) = R_M \oplus I_M.
\end{equation}
By Lemma \ref{lem:tr-rej} we have
$$
\rej_M(P_{d_2}) = \rej_{P_M \ds R_M \ds I_M}(P_{d_2}) = \rej_{P_M}(P_{d_2}) \ds
\rej_{R_M}(P_{d_2}) \ds \rej_{I_M}(P_{d_2}).
$$
By Theorem \ref{thm:Kronecker}(3) we have $\Hom_A(R_M, P_{d_2}) = 0$ and
$\Hom_A(I_M, P_{d_2}) = 0$, which shows that
$$
\rej_{R_M}(P_{d_2}) = R_M \quad\text{and}\quad  \rej_{I_M}(P_{d_2}) = I_M.
$$
If a preprojective indecomposable module $P_i$ is a direct summand of $M$,
then it follows from $(i-1,i) = \udim P_i \le \udim M = (d_1,d_2)$ that $i \le d_2$
(see Remark \ref{rmk:order}(2)).
Therefore, we have $P_M = \Ds_{i=1}^{d_2}P_i^{(a_i)}$ for some $a_i \ge 0$
(we identify $P_i$ with $F(P_i)$), and then
$\rej_{P_M}(P_{d_2})= \Ds_{i=1}^{d_2}(\rej_{P_i}(P_{d_2}))^{(a_i)}$.
Now if $i \le d_2$, then by Remark \ref{rmk:order}(1)
we have a monomorphism $P_i \to P_{d_2}$, which shows that
$\rej_{P_i}(P_{d_2}) = 0$ for all $i \le d_2$, and therefore
$$\rej_{P_M}(P_{d_2})= 0.$$
Hence the equality \eqref{eq:RM-IM} holds.
\end{proof}

\begin{proposition}[Calculation of $I_M$]\label{prp:IM}
If $\{g_1,\dots,g_b\}$ is a basis of\/ $\Hom_A(I_{d_1},R_M \oplus I_M)$,
then we have
$$
\sum^b_{i=1}\im g_i = I_M.
$$
\end{proposition}

\begin{proof}
By assumption it is obvious that $\sum^b_{i=1}\im g_i  = \tr_{R_M \ds I_M}(I_{d_1})$.
Therefore it is enough to show that
\begin{equation}\label{eq:IM}
\tr_{R_M\ds I_M}(I_{d_1}) = I_M.
\end{equation}
By Lemma \ref{lem:tr-rej} we have
$$
\tr_{R_M\ds I_M}(I_{d_1}) = \tr_{R_M}(I_{d_1}) \ds \tr_{I_M}(I_{d_1}).
$$
By Theorem \ref{thm:Kronecker}(3) we have $\Hom_A(I_{d_1}, R_M) = 0$, which shows that
$$
\tr_{R_M}(I_{d_1}) = 0.
$$
If a preinjective indecomposable module $I_i$ is a direct summand of $M$, then
 it follows from $(i,i-1) = \udim I_i \le \udim M = (d_1,d_2)$ that $i \le d_1$.
Therefore we have $I_M = \Ds_{i=1}^{d_1}I_i^{(b_i)}$ for some $b_i \ge 0$
(we identify $I_i$ with $F(I_i)$), and then
$\tr_{I_M}(I_{d_1})= \Ds_{i=1}^{d_1}(\tr_{I_i}(I_{d_1}))^{(b_i)}$.
Now if $i \le d_1$, then we have an epimorphism $I_{d_1} \to I_i$, which shows that
$\tr_{I_i}(I_{d_1}) = I_i$ for all $i \le d_1$, and therefore
$$\tr_{I_M}(I_{d_1})= I_M.$$
Hence the equality \eqref{eq:IM} holds.
\end{proof}

By Propositions \ref{prp:RM-IM} and \ref{prp:IM} we have the following.

\begin{proposition}[Calculation of $R_M$]\label{prp:RM}
Let $\{f_1,\dots,f_a\}$ a basis of\/ $\Hom_A(M,P_{d_2})$ and
$\{g_1,\dots,g_b\}$ a basis of\/ $\Hom_A(I_{d_1},\bigcap^a_{i=1}\Ker f_i)$.
Then we have
$$
R_M \iso \left(\bigcap^a_{i=1}\Ker f_i\right)/\left(\sum^b_{i=1}\im g_i\right).
$$
\end{proposition}
By this isomorphism we identify $R_M$ with the right hand side.
Since $R_M = (R_M (\al),R_M(\be))$ is the direct sum of regular indecomposable modules, both $R_M(\al)$ and $R_M(\be)$ are square matrices, say of size $d$.
Put $R(\infty):=\tr_{R_M}(R_d(\infty))$. 
Note that $\tr_{R_M}(R_d(\infty)) = \tr_{R_M}(\Ds_{n=1}^d R_n(\infty))$ because there exists an epimorphism $R_n (\infty) \to R_m (\infty)$ for $n\geq m$. 
Then $R_M = R(\infty) \ds R'$ for some $A$-submodule $R'=(X',Y')$ of $R_M$ such that $R'$ has no direct summand of the form $R_n (\infty)$ for any $n$
by Theorem \ref{thm:Kronecker}(3)(iii).
[Decompose $R_M$ into indecomposables of the form $R_n(\la)$ with $n \ge 1, \la \in \bbP^1(\Bbbk)$.
Then $R'$ is given by the direct sum of those direct summands of the form $R_n(\la)$ with $\la \ne \infty$ because $R(\infty)$ is given by the direct sum of summands
of the form $R_n(\infty)$. Note that $R'$ is also computed as $R' = \rej_{R_M}(R_d(\infty))$.]
Since the matrix $X'$ is invertible, 
we have $$R' \iso (E_l,(X')^{-1}Y')$$
for some $l \leq d$. 
Therefore, the set  $\La$  of eigenvalues of $(X')^{-1}Y'$ is finite.

Then by Propositions \ref{prp:RM-IM}, \ref{prp:IM} and \ref{prp:RM}, we obtain the following.

\begin{theorem}
Set $$
S_M:=\{P_i,I_j,R_k(\la)\mid 1\le i \le d_2, 1\le j \le d_1,1\le k \le d, \la\in\La\cup\{\infty\}\}.$$
Then this gives a solution to the problem {\rm (II)} for the Kronecker algebra.
\end{theorem}

\begin{remark}
Note that if $R(\infty) = 0$, then we can replace $S_M$ by
$$
\{P_i,I_j,R_k(\la)\mid 1\le i \le d_2, 1\le j \le d_1,1\le k \le d, \la\in\La\}.$$

\end{remark}

\section{Examples for the Kronecker algebra}
$(1)$
For a preprojective module $M=P_3=\left( \begin{bmatrix}1&0\\ 0&1 \\0&0\end{bmatrix}, \begin{bmatrix}0&0\\1&0\\ 0&1 \end{bmatrix}\right)$ with $\udim M = (2,3)$, 
we will compute $p_n (M)\ (n\in \bbN)$ and then we will give $\boldsymbol{d}_M (P_n) \ (n \in \bbN)$.
By Definition \ref{dfn:rank-mat} we have
$p_1 (M)=0$,
$$
\begin{gathered}
p_2 (M) 
=\rank 
\left[ 
    \begin{array}{c|c}
    M(\be) &M(\al) 
    \end{array}
\right]
=\rank 
\left[ 
\vphantom{\begin{array}{c}1\\1\\1\\\end{array}}
    \begin{array}{cc|cc}
          0&0&1&0 \\
          1&0&0&1 \\
          0&1&0&0
    \end{array}
\right]=3\text{,}
\\
p_3 (M)
=\rank
\left[ 
\vphantom{\begin{array}{c}1\\1\\\end{array}}
    \begin{array}{c|c|c}
    M(\be) &M(\al) &0 \\
    \hline   
    0&M(\be)&M(\al)
    \end{array}
\right]
=\rank
\left[ 
\vphantom{\begin{array}{c}1\\1\\\end{array}}
    \begin{array}{c|c|c}
    \Mb&\Ma & \\
    \hline
    &\Mb&\Ma
    \end{array}
\right]=6\text{,}
\\
p_4 (M)
=\rank
\left[ 
\vphantom{\begin{array}{c}1\\1\\1\\\end{array}}
    \begin{array}{c|c|c|c}
    M(\be) &M(\al) &0 &0 \\
    \hline   
    0&M(\be)&M(\al) &0 \\
    \hline
    0&0&M (\be)&M(\al) 
    \end{array}
\right]
=\rank
\left[ 
\vphantom{\begin{array}{c}1\\1\\\end{array}}
    \begin{array}{c|c|c|c}
    \Mb&\Ma & \\
    \hline
    &\Mb&\Ma \\
    \hline
    &&\Mb&\Ma
    \end{array}
\right]=8,
\end{gathered}
$$
and
$$
p_5 (M)
=\rank
\left[ 
\vphantom{\begin{array}{c}1\\1\\1\\\end{array}}
    \begin{array}{c|c|c|c|c}
    M(\be) &M(\al) &0 &0 &0\\
    \hline   
    0&M(\be)&M(\al) &0 &0\\
    \hline
    0&0&M (\be)&M(\al) &0 \\
    \hline
    0&0&0& M (\be)&M(\al) 
    \end{array}
\right]
$$
$$
=\rank
\left[ 
\vphantom{\begin{array}{c}1\\1\\\end{array}}
    \begin{array}{c|c|c|c|c}
    \Mb&\Ma & & \\
    \hline
    &\Mb&\Ma &\\
    \hline
    &&\Mb&\Ma&\\
    \hline
    &&&\Mb&\Ma
    \end{array}
\right]=10\text{.}
$$
Similarly, we have $p_n (M) =2 n$ for $n \geq 3$.
Hence by Theorem \ref{thm:pd} we have
$$
\begin{array}{l}
\boldsymbol{d}_M (P_1) = 3- p_2 (M) = 0 ,\\
\boldsymbol{d}_M (P_2) = 2 p_2 (M) -p_1(M) - p_3 (M) = 6 -0 - 6=0, \\
\boldsymbol{d}_M (P_3) = 2 p_3 (M) -p_2(M) - p_4 (M) = 12 -3 - 8=1,\\
\end{array}
$$
and for $n \geq 4$, 
$$
\boldsymbol{d}_M (P_n) = 2 p_n (M) -p_{n-1}(M) - p_{n+1} (M) = 2 \cdot 2n -2(n-1) - 2(n+1) =0. 
$$
Thus we can confirm $\boldsymbol{d}_M (P_3) = 1$ and $\boldsymbol{d}_M (P_n)=0$ for $n\not=3$. 

$(2)$
For a module $M=\left( \begin{bmatrix}0&0 \\ 1 & 0\end{bmatrix},0_{2,2} \right)= P_1 \ds R_1 (0) \ds I_1$ with $\udim M=(2,2)$, 
we will compute $\rej_M(P_{2})$ and $\tr_{\rej_M(P_{2})}(I_{2})$. 
Recall that $P_2 = \left( \begin{bmatrix}1 \\0\end{bmatrix}, \begin{bmatrix}0\\ 1\end{bmatrix} \right)$. 
If $(X,Y) \in \Hom_A (M,P_2)$, then we have
$X =0_{1,2}, Y=\bmat{a&0 \\ b & 0}$ for some $a, b \in \Bbbk$,
and we can take
$\left\{ f_1=  \left(0_{1,2}, \begin{bmatrix}1&0\\ 0&0\end{bmatrix} \right) , f_2 =\left(0_{1,2}, \begin{bmatrix}0&0\\ 1&0\end{bmatrix} \right) \right\}$ as a basis of $\Hom_A (M,P_2)$. 
Hence we have
$$
\rej_M(P_{2})=\Ker f_1 \cap \Ker f_2 = \left( \begin{bmatrix}1 & 0\end{bmatrix},0_{1,2} \right)= R_1 (0) \ds I_1
$$
with $\udim \rej_M(P_{2}) =(2,1)$ 
and have $M/\rej_M(P_{2}) \iso P_1$. 
Moreover, recall that $I_2 = \left( \begin{bmatrix}1 & 0\end{bmatrix},\begin{bmatrix}0 & 1\end{bmatrix} \right)$. 
If $(X,Y)\in \Hom_A (I_2, \rej_M(P_{2}))$, then we have  $X =\begin{bmatrix}0&0 \\ c & d\end{bmatrix}, Y=0_{1,1}$ for some $c, d \in \Bbbk$,
and we can also take $\left\{ g_1=   \left(\begin{bmatrix}0&0\\ 1&0\end{bmatrix},0_{1,1} \right), g_2=\left(\begin{bmatrix}0&0\\ 0&1\end{bmatrix}, 0_{1,1} \right) \right\}$ as a basis of $\Hom_A (I_2, \rej_M(P_{2}))$. 
Therefore, we have
$$
\tr_{\rej_M(P_{2})}(I_{2}) = \im g_1 +\im g_2 = 
\left(\J_{0,1},\J_{0,1} \right) = I_1
$$
with $\udim \tr_{\rej_M(P_{2})}(I_{2}) = (1,0)$,  
and $\rej_M(P_{2})/ \tr_{\rej_M(P_{2})}(I_{2}) \iso R_1(0) $.
Thus we can confirm the process to get $S_M=\{P_1,P_2,I_1,I_2,R_1(0)\}$ in Section $5$.

\section*{Acknowledgments}
The starting idea was made when the first named author was at 
the lecture on the Kronecker canonical form of singular mixed matrix pencils by S.~Iwata at CREST meeting in February, 2016. 
In fact, later we recognized that
Iwata--Shimizu \cite{IS} had given relationships between
values of ranks $p_n(M)$ (resp.\ $i_n(M), r_n(0, M)$, and $r_n(\infty, M)$)
in Definition \ref{dfn:rank-mat} and the values of $\boldsymbol{d}_M(L)$ for $L = P_n$
(resp.\ $I_n, R_n(0)$, and $R_n(\infty)$) in Theorem \ref{thm:Kronecker}.
However, they did not mention regular indecomposables
$R_n(\la)$ with $\la \ne 0, \infty$.
The authors would like to thank S.~Iwata for his lecture and M.~Takamatsu for informing us of
the paper \cite{IS}.
After submitting the paper E.~Escolar and the referee pointed out that the paper
\cite{DM} is already published dealing with the same problem with similar results,
for which we are very thankful.
This work is partially supported by Grant-in-Aid for Scientific Research 25610003 and 25287001 from JSPS (Japan Society for the Promotion of Science), and by JST (Japan Science and Technology Agency) CREST Mathematics (15656429).

\bibliographystyle{plain}

\end{document}